\documentclass[12pt,leqno]{amsart}
\newtheorem{theorem}{Theorem}[section]
\newtheorem{definition}{Definition}[section]
\newtheorem{lemma}{Lemma}[section]
\newtheorem{corollary}{Corollary}[section]

\newtheorem{remark}{Remark}[section]

\newtheorem{example}{Example}[section]

\newcommand{\be}{\begin{equation}}
\newcommand{\ee}{\end{equation}}
\numberwithin{equation}{section}
\newcommand{\bea}{\begin{eqnarray}}
\newcommand{\eea}{\end{eqnarray}}
\newcommand{\beb}{\begin{eqnarray*}}
\newcommand{\eeb}{\end{eqnarray*}}
\usepackage{amssymb,amsfonts,amsthm,setspace,indentfirst}
\usepackage[dvips]{graphics}
\usepackage{epsfig}
\begin{document}
\title{Statistical and rough statistical convergence in an S-metric space}
\author{Sukila Khatun$^{1}$ and Amar Kumar Banerjee$^{2}$}
\address{$^{1}$Department of Mathematics, The University of Burdwan,
Golapbag, Burdwan-713104, West Bengal, India.} 
\address{$^{3}$Department of Mathematics, The University of Burdwan,
Golapbag, Burdwan-713104, West Bengal, India.}
\email{$^{1}$sukila610@gmail.com}
\email{$^{3}$akbanerjee@math.buruniv.ac.in, akbanerjee1971@gmail.com}

\begin{abstract}
In this paper, using the concept of natural density, we have introduced the ideas of statistical and rough statistical convergence in an $S$-metric space. We have investigated some of their basic properties. We have defined statistical Cauchyness and statistical boundedness of sequences and then some results related these ideas have been studied. We have defined the set of rough statistical limit points of a sequence in an $S$-metric space and have proved some relevant results associated with such type of convergence. 
\end{abstract}
\subjclass[2020]{40A05, 40A99.}
\keywords{Natural density, statistical convergence, rough convergence, S-metric spaces, rough limit sets.}
\maketitle
\section{\bf{Introduction}}
The notion of statistical convergence is a generalization of ordinary convergence. The definition of statistical convergence was introduced by H. Fast \cite{HF} and H. Steinhaus \cite{HS} independently in the year of 1951. The formal definitions are as follows: \\
For $ B \subset \mathbb{N}$ and $B_{n}= \{k \in B : k \leq n \}$.  Then natural density of $B$ is denoted by $\delta(B)$ and defined by $\delta(B)=lim_{n\to\infty} \frac{|B_{n}|}{n}$, if the limit exists, where $|B_{n}|$ stands for the cardinality of $B_{n}$. 
A real sequence $\{ \xi_n \}$ is said to be statistically convergent to $\xi$ if for every $\varepsilon>0$ the set $B(\varepsilon)= \{ k \in \mathbb{N}: |\xi_{n} - \xi| \geq \varepsilon \}$ has natural density zero. In this case, $\xi$ is called the statistical limit of the sequence $\{ \xi_n \}$ and we write $st-lim \ \xi_{n}= \xi$.
In many directions works were done on statistical convergence by many authors \cite{{RA},{DG1},{DG2},{LLG},{PMROUGH2},{TS}}. \\
Many authors tried to give generalization of the concept of metric spaces in several ways. For example probabilistic metric spaces \cite{KM}, $C^*$-algebra valued metric spaces  \cite{MJS} etc. and related works were carried out by several authors e.g. \cite{{AKBAB},{AKB,AP}, {AKB,MP}}.
In 2012, Sedghi et al. \cite{SSA} introduced the idea of $S$-metric spaces as a new structure of metric spaces. \\
 The idea of rough convergence of sequences in a normed linear space was introduced by H. X. Phu \cite{PHU} in 2001. Also, Phu \cite{PHU1} extended this concept in an infinite dimensional normed space in 2003. After that several works \cite{{AYTER1},{AYTER2},{NH},{PMROUGH1},{SUK1}} were done in many generalized spaces. For example the idea of rough convergence in a metric space was studied by S. Debnath and D. Rakhshit \cite{DR}, in a cone metric space it was studied by A. K. Banerjee and R. Mondal \cite{RMROUGH} and in a partial metric space by A. K. Banerjee and S. Khatun \cite{{SUK2},{SUK4}}.\\
 In this paper, we have introduced the idea of statistical convergence in an $S$-metric space. We have investigated some basic properties of statistical convergence in this space. We have defined statistical Cauchyness and statistical boundedness. Also we have introduced the idea of rough statistical convergence which is an extension work of rough convergence in an $S$-metric space. We have defined the rough limit set of a sequence and investigated some basic properties of rough limit set in this space. Also we have proved some relevant theorems in this space.\\
\section{\bf{Preliminaries}}

\begin{definition} \cite{SSA}
In a nonempty set $X$ a function $S: X^{3} \longrightarrow [0, \infty)$ is said to be an $S$-metric on $X$ if the following three conditions hold, for every $x,y,z,a \in X$:\\
$(i)$ $S(x,y,z) \geq 0$,\\
$(ii)$ $S(x,y,z) = 0$ if and only if $x=y=z$,\\
$(iii)$  $S(x,y,z) \leq S(x,x,a) + S(y,y,a) + S(z,z,a)$.\\
Then the pair $(X, S)$ is said to be an $S$-metric space.    
\end{definition}
Properties and examples of $S$-metric spaces have been thoroughly discussed in \cite{SSA}.

\begin{definition} \cite{SSA}
In an $S$-metric space  $(X,S)$ for $r>0$ and $x \in X$ we define the open and closed ball of radius $r$ and center $x$ respectively as follows  : 
\begin{center}
    $B_{S}(x,r)=\{ y \in X : S(y,y,x)<r  \}$ \\
$B_{S}[x,r]=\{ y \in X : S(y,y,x) \leq r  \}$ 
\end{center}
\end{definition}
In an $S$-metric space $(X,S)$ the collection of all open balls $\tau$ in $X$ forms a base of a topology on $X$ called topology induced by the $S$-metric. The open ball $B_{S}(x,r)$ of radius $r$ and centre $x$ is an open set in $X$.

\begin{definition} \cite{SSA}
In an $S$-metric space $(X,S)$ a subset $A$ of $X$ is said to be $S$-bounded if there exists a $r>0$ such that $S(x,x,y)<r$ for every $x,y \in A$.
\end{definition}

\begin{definition} \cite{SSA}
In an $S$-metric space $(X,S)$ a sequence $\{x_{n} \}$ is said to be convergent to $x$ in $X$ if for every $\epsilon > 0$ there exists a natural number $k$ such that $S(x_{n}, x_{n}, x)< \epsilon$ for every $n \geq k$.
\end{definition}

\begin{definition} \cite{SSA}
In an $S$-metric space $(X,S)$ a sequence $\{x_{n} \}$ is said to be a Cauchy sequence in $X$ if for every $\varepsilon > 0$ there exists a natural number $k$ such that $S(x_{n}, x_{n}, x_{m})< \varepsilon$ for every $n,m \geq k$.
\end{definition}

\begin{lemma} \cite{SSA} 
 In an $S$-metric space, we have $S(x,x,y)= S(y,y,x)$.   
\end{lemma}

\begin{definition} \cite{SUK1}
A sequence $\{ x_{n} \}$ in an $S$-metric space $(X, S)$ is said to be rough convergent or simply $r$-convergent to $p$ if for every $\varepsilon > 0$ there exists a natural number $k$ such that $S(x_{n}, x_{n}, p) <r + \varepsilon $ holds for all $n \geq k$.\\
$r$ is called the roughness degree of $\{x_{n}\}$. Note that rough limit of a sequence may not be unique. The set of all rough limits of a sequence is denoted by $LIM^{r}x_{n}$. If $r=0$ then the idea of rough convergence reduces to ordinary convergence.
\end{definition}

\section{\bf{Statistical convergence in an $S$-metric space}}

\begin{definition}
    A sequence $\{x_n\}$ in an $S$-metric space $(X,S)$ is said to be statistically convergent to $x \in X$ if for any $\varepsilon>0$, $\delta(A(\varepsilon))=0$, where $A(\varepsilon)=\{n \in \mathbb N: S(x_n,x_n,x) \geq \varepsilon\}$ and we write $st-lim_{n\to\infty}x_n=x$.     
\end{definition}

\begin{theorem}
Every  convergent sequence is statistically convergent in an $S$-metric space $(X,S)$.
\end{theorem}
\begin{proof}
Let $(X,S)$ be an $S$-metric space and let $\{x_n\}$ be a convergent sequence converging to $x$ in $X$.
Then for $\varepsilon>0$, there exists $N \in \mathbb N$ such that $ S(x_n,x_n,x)< \varepsilon$, $\forall n \geq N$.
We consider the set $A(\varepsilon)=\{n \in \mathbb N: S(x_n,x_n,x) \geq \varepsilon\} \subset \{1,2,3,......,(N-1) \}=P$(say).
Since $P$ is a finite set, so $\delta(P)=0$.
Hence $\delta(A(\varepsilon))=0$.
Therefore $\{x_n\}$ is statistically convergent to $x$ in $(X,S)$.
So $ st-lim_{n\to\infty}x_n=x $. 
    
\end{proof}

\begin{remark}
    The converse of the above Theorem may not be true as shown in the following example.
\end{remark}

\begin{example}
Let $X=\mathbb R^2$ and $||.||$ be the Euclidean norm on $X$, then $S(x,y,z)=||x-z||+||y-z||$, $\forall x,y,z \in X$ is an $S$-metric space on $X$.\\ 
Let $\{x_{n}\}$ be a sequence in $\mathbb R^2$ defined by 
  \begin{equation*}
    \  x_{n}= \begin{cases}
          (k,k) & \text { if $n=k^2$ for some $k \in \mathbb N$}, \\
          (0,0) & \text { otherwise }
                 \end{cases}
  \end{equation*}
Let $\varepsilon>0$ be given.
Then we will show that $\{x_{n}\}$ is statistically convergent to $x=(0,0)$.\\
Now, $A(\varepsilon)=\{n \in \mathbb N: S(x_n,x_n,x)\geq \varepsilon\} \subset \{1^2,2^2,3^2.......\}=P$ (say).\\
As $\delta(P)=0$, therefore $\delta(A(\varepsilon))=0$.
Hence $ st-lim_{n\to\infty}x_n=x $.\\
So, $\{x_{n}\}$ is statistically convergent to $x=(0,0)$.\\
If $\{x_{n}\}$ were bounded, then $\exists \ B(>0) \in \mathbb R$ such that $S(x_{n},x_{n},x_{m})< B$ $\forall n,m \in \mathbb N$ i.e. $2||x_n-x_m||<B$ $\forall n,m \in \mathbb N$ i.e. $||x_n-x_m||< \frac{B}{2}$......(1) $\forall n,m \in \mathbb N$.
Choose $n=k^2$ where $k>B$, $k \in \mathbb N$ and $m$ is any natural number such that  $m \neq k^2$ for $k \in \mathbb N$, then $x_m=(0,0)$ and $x_n=(k,k)$.
So $||x_n-x_m||=\sqrt{(k-0)^2+(k-0)^2}=\sqrt{2}k>B>\frac{B}{2}$. This contradicts (1).
So it is not bounded and hence $\{x_{n}\}$ is not ordinary convergent.
\end{example}

By the following theorem we conclude that the statistical limit in an $S$-metric space is unique.

\begin{theorem}
    Let $\{x_n\}$ be a sequence in an $S$-metric space $(X,S)$ such that $x_{n} \stackrel{st}{\longrightarrow} x$ and $x_{n} \stackrel{st}{\longrightarrow} y$, then $x=y$.
\end{theorem}

\begin{proof}
    Let $\varepsilon>0$ be arbitrary. Since $x_{n} \stackrel{st}{\longrightarrow} x$ and $x_{n} \stackrel{st}{\longrightarrow} y$,
     for $\varepsilon>0$, $\delta(A_1(\varepsilon))=0$ and $\delta(A_2(\varepsilon))=0$, where $A_1(\varepsilon)=\{n \in \mathbb N: S(x_n,x_n,x) \geq \frac{\varepsilon}{3}\}$ and $A_2(\varepsilon)=\{n \in \mathbb N: S(x_n,x_n,y) \geq \frac{\varepsilon}{3}\}$.
    Let $K(\varepsilon)=A_1(\varepsilon) \cup A_2(\varepsilon)$, then $\delta(K(\varepsilon))=0$. Hence $\delta(K(\varepsilon))^c=1$. Suppose $ k \in (K(\varepsilon))^c=(A_1(\varepsilon) \cup A_2(\varepsilon))^c=(A_1(\varepsilon))^c \cap (A_2(\varepsilon))^c$.
    Then 
    \begin{equation*}
         \begin{split}
    S(x,x,y) & \leq S(x,x,x_k)+S(x,x,x_k)+S(y,y,x_k) \\
             & =S(x_k,x_k,x)+S(x_k,x_k,x)+S(x_k,x_k,y) \\
             & < \frac{\varepsilon}{3}+\frac{\varepsilon}{3}+\frac{\varepsilon}{3} = \varepsilon 
        \end{split}
    \end{equation*}
    Since $\varepsilon>0$ is arbitrary, we get $S(x,x,y)=0$.
    Therefore $x=y$.
\end{proof}

\begin{definition}
    A sequence $\{x_n\}$ in an $S$-metric space $(X,S)$ is said to be statistically Cauchy if for any $\varepsilon>0$, $\exists N \in \mathbb N$ such that $\delta(B(\varepsilon))=0$, where $B(\varepsilon)=\{n \in \mathbb N: S(x_n,x_n,x_N) \geq \varepsilon\}$.
\end{definition}

\begin{theorem}
    Let $\{x_n\}$ be a statistically convergent sequence in an $S$-metric space $(X,S)$. Then the sequence $\{x_n\}$ is statistically Cauchy sequence in $(X,S)$. 
\end{theorem}

\begin{proof}
Let the sequence $\{x_n\}$ be statistically convergent to $x$ in $(X,S)$.
Then for an arbitrary $\varepsilon>0$, $\delta(A(\varepsilon))=0$, where $A(\varepsilon)=\{n \in \mathbb N: S(x_n,x_n,x)\geq \frac{\varepsilon}{3}\}$. So $ \delta{(A(\varepsilon))}^c=1$.
 Take $N_\varepsilon \in A(\varepsilon)^c$. Then $ S(x_{N_\varepsilon}, x_{N_\varepsilon}, x)< \frac{\varepsilon}{3}$.
 Now, we show that 
 \begin{equation*}
     \{n \in \mathbb N: S(x_n,x_n,x) < \frac{\varepsilon}{3}\} \subset \{n \in \mathbb N: S(x_n,x_n,x_{N_\varepsilon}) < \varepsilon\}.
 \end{equation*}
 Let $k \in \{n \in \mathbb N: S(x_n,x_n,x) < \frac{\varepsilon}{3}\}=A(\varepsilon)^c$.
 Then $S(x_k,x_k,x) < \frac{\varepsilon}{3}$.
 Now, \begin{equation*}
     \begin{split}
       S(x_k,x_k,x_{N_\varepsilon}) & \leq S(x_k,x_k,x)+S(x_k,x_k,x)+S(x_{N_\varepsilon},x_{N_\varepsilon},x) \\
       & < \frac{\varepsilon}{3}+\frac{\varepsilon}{3}+\frac{\varepsilon}{3}  = \varepsilon
     \end{split}
 \end{equation*}
 So, $k \in \{n \in \mathbb N: S(x_n,x_n,x_{N_\varepsilon}) < \varepsilon\}$.
 Hence $A(\varepsilon)^c \subset B(N_\varepsilon)^c$, where $B(N_\varepsilon)= \{n \in \mathbb N: S(x_n,x_n,x_{N_\varepsilon})\geq \varepsilon\}$.
This implies that $B(N_\varepsilon) \subset A(\varepsilon)$.
Since $\delta(A(\varepsilon))=0$, so $\delta(B(N_\varepsilon))=0$.
Hence the theorem follows.
\end{proof}

\begin{theorem}
    Let $\{x_n\}$ be a statistical convergent sequence in an $S$-metric space $(X,S)$. Then there is a convergent sequence $\{y_n\}$ in $X$ such that $x_n=y_n$ for almost all $n \in \mathbb N$.
\end{theorem}
\begin{proof}
    Let $\{x_n\}$ be a statistical convergent to $x$.
    Then for any $\varepsilon>0$, we have 
    $\delta(\{n \in \mathbb N: S(x_n,x_n,x)\geq \varepsilon \})=0$.
    So, $\delta(\{n \in \mathbb N: S(x_n,x_n,x) < \varepsilon \})= lim_{n\to\infty} \frac{| \{ k \leq n:S(x_k,x_k,x) < \varepsilon \}|}{n}=1$.
    So, for every $k \in \mathbb N$, $\exists$ $n_k \in \mathbb N$ such that $\forall$ $n > n_k$
    \begin{equation*}
        \frac{| \{ k \leq n:S(x_k,x_k,x) < \frac{1}{2^k} \}|}{n} > 1- \frac{1}{2^k}
    \end{equation*}
    Choose \begin{equation*}
    \  y_{m}= \begin{cases}
          x_m, & \text { if $1 \leq m \leq n_1$ }, \\
          x_m, & \text { if $n_k<m \leq n_{k+1}, S(x_m,x_m,x)< \frac{1}{2^k}$}, \\
          x,   & \text { otherwise }
                 \end{cases}
  \end{equation*}
  Let $\varepsilon>0$. Choose $k \in \mathbb N$, so that $ \frac{1}{2^k}< \varepsilon$.
  Now, for each $m>n_k$, $S(y_m,y_m,x)=S(x_m,x_m,x)< \frac{1}{2^k}< \varepsilon$.
  Hence $\{y_m\}$ converges to $x$. So $lim_{n\to\infty}y_m=x$. \\
  Let $n \in \mathbb N$ be fixed and let $n_k< n \leq n_{k+1}$, then we have 
  \begin{equation*}
      \{m \leq n:y_m \neq x_m \} \subset \{ 1,2,.....n\} - \{ m \leq n:S(x_m,x_m,x)< \frac{1}{2^k} \}
  \end{equation*}
  So, \begin{equation*}
    \begin{split}
 \frac{1}{n} |\{m \leq n:y_m \neq x_m \}| &  \leq 1- \frac{1}{n} |\{ m \leq n:S(x_m,x_m,x)< \frac{1}{2^k} \}| \\
                                          & < \frac{1}{2^k}< \varepsilon
          \end{split}
  \end{equation*}
  Hence $lim_{n\to\infty} \frac{1}{n} |\{m \leq n:y_m \neq x_m \}|=0$.
  So, $\delta(\{m \in \mathbb N:y_m \neq x_m \} )=0$.
  Therefore $x_m=y_m$ for almost all $m \in \mathbb N$.
\end{proof}

\begin{corollary}
 Every statistical convergent sequence has a convergent subsequence.   
\end{corollary}
\begin{proof}
  Let $\{x_n\}$ be a statistical convergent sequence.
  So $\exists$ a convergent sequence $\{y_n\}$ such that $x_n=y_n$ almost everywhere i.e $\delta(A)=0$, where $A=\{m \in \mathbb N:y_m \neq x_m \}$.
  Let us enumerate the set $\mathbb N \setminus A$ by $\{n_1<n_2<..... \}$.
  Therefore $y_{n_{k}}=x_{n_{k}}$ $\forall$ $k=1,2,3......$ and since $\{y_{n_{k}}\}$ is convergent, $\{x_{n_{k}}\}$ is a convergent subsequence of $\{x_n\}$.
\end{proof}

\begin{definition}
  Let $(X,S)$ be an $S$-metric space. If every statistically Cauchy sequence is statistically convergent, then $(X,S)$ is called statistically complete.  
\end{definition}

\begin{theorem}
    If $\{x_n\}$ is Cauchy then it is statistically Cauchy in an $S$-metric space. 
\end{theorem}
\begin{proof}
Let $\{x_n\}$  be a Cauchy sequence in $(X,S)$.
Let $\varepsilon>0$.
Then there exists a natural number $k$ such that $S(x_n,x_n,x_m)< \varepsilon$ for every $n,m \geq k$.
So, in particular  $S(x_n,x_n,x_k)< \varepsilon, \ \forall n\geq k$.
So, the set $B(\varepsilon)=\{n \in \mathbb N: S(x_n,x_n,x_k) \geq \varepsilon\} \subset \{1,2,3,......,(k-1) \}=Q$(say).
Since $Q$ is finite, so $\delta(Q)=0$.
Hence $\delta(B(\varepsilon))=0$.
Therefore $\{x_n\}$ is statistically Cauchy.
\end{proof}

\begin{theorem}
  Every statistically complete $S$-metric space is complete.  
\end{theorem}
\begin{proof}
    Let $(X,S)$ be a statistically complete $S$-metric space.
    Suppose that $\{x_n\}$ is a Cauchy sequence in $(X,S)$, then it is statistically Cauchy sequence in $(X,S)$.
    Since $(X,S)$ is statistically complete, so $\{x_n\}$ is statistically convergent.
    By the Corollary 2.1, there is a subsequence $\{x_{n_{k}}\}$ of $\{x_n\}$ such that $\{x_{n_{k}}\}$ converges to a point $x \in X$.
    Since $\{x_n\}$ is Cauchy, hence for $\varepsilon>0$, there exists $N_1 \in \mathbb N$ such that $S(x_n,x_n,x_m)< \frac{\varepsilon}{3}$ for $n,m \geq N_1$.
    On the other hands, $\{x_{n_{k}}\}$ converges to $x$, so there exists $k_0$ such that $S(x_{n_{k}},x_{n_{k}},x)< \frac{\varepsilon}{3}$ \ $\forall$ $k \geq k_0$. 
    So in particular $S(x_{n_{k_0}},x_{n_{k_0}},x)< \frac{\varepsilon}{3}$ and let $N=max\{N_1,n_{k_0} \}$.
    Then for $n \geq N$, we have 
    \begin{equation*}
        \begin{split}
    S(x_n,x_n,x) & \leq S(x_n,x_n,x_{n_{k_0}})+S(x_n,x_n,x_{n_{k_0}})+S(x,x,x_{n_{k_0}}) \\
            & =S(x_n,x_n,x_{n_{k_0}})+S(x_n,x_n,x_{n_{k_0}})+S(x_{n_{k_0}},x_{n_{k_0}},x) \\
            & <  \frac{\varepsilon}{3}+\frac{\varepsilon}{3}+\frac{\varepsilon}{3} = \varepsilon
        \end{split}
    \end{equation*}
Hence $lim_{n\to\infty}x_n=x $.
So $(X,S)$ is a complete $S$-metric space.
\end{proof}

\begin{definition} (cf.\cite{AKB})
A sequence $\{x_n\}$ in an $S$-metric space $(X,S)$ is said to be statistically bounded if for any fixed $u \in X$ there exists a positive real number $B$ such that 
\begin{equation*}
    \delta(\{n \in N:S(x_n,x_n,u) \geq B\})=0
\end{equation*}
\end{definition}

\begin{theorem}
  Every statistically Cauchy sequence is statistically bounded in an $S$-metric space $(X,S)$. 
\end{theorem}
\begin{proof}
   Let $\varepsilon>0$ and let $\{x_n\}$ be a statistically Cauchy sequence in $(X,S)$. 
   Then there exists $N \in \mathbb N$ such that $\delta(B)=0$, where $B=\{n \in N:S(x_n,x_n,x_N) \geq \frac{\varepsilon}{2}\}$.
   Let $k \in B^c$, then $S(x_k,x_k,x_N) < \frac{\varepsilon}{2}$.
   Let $u \in X$ be fixed, then for the same $\varepsilon>0$, we have 
   \begin{equation*}
       \begin{split}
           S(x_k,x_k,u) & \leq S(x_k,x_k,x_N)+S(x_k,x_k,x_N)+S(u,u,x_N) \\
                        & < \frac{\varepsilon}{2}+\frac{\varepsilon}{2}+S(x_N,x_N,u) \\
                        & = \varepsilon+S(x_N,x_N,u) \\
                        & =a \ \text{(say)}
       \end{split}
   \end{equation*}
So, $k \in \{n \in N:S(x_n,x_n,u)<a\}$.
Therefore $B^c \subset \{n \in N:S(x_n,x_n,u)<a\}$.
Since $\delta(B)=0$, $\delta(B^c)=1$.
Hence $\delta(\{n \in N:S(x_n,x_n,u)<a\})=1$.
This implies that $\delta(\{n \in N:S(x_n,x_n,u) \geq a\})=0$.
Therefore the sequence $\{x_n\}$ is statistically bounded in $X$. 
\end{proof}

\begin{corollary}
 Every statistically convergent sequence is statistically bounded in an $S$-metric space.   
\end{corollary}

\section{\bf{Rough statistical convergence in an $S$-metric space}}

\begin{definition}
    A sequence $\{x_n\}$ in an $S$-metric space $(X,S)$ is said to be rough statistical convergent or r-statistical convergent (or in short r-st convergent) to $x$, if for every $\varepsilon>0$ and for some roughness degree $r$,  $\delta(\{n \in \mathbb N : S(x_n,x_n,x) \geq r+\varepsilon\})=0$.
\end{definition}
For $r=0$, the rough statistical convergent becomes the statistical convergent in any $S$-metric spaces $(X,S)$. If a sequence $\{x_n\}$ is rough statistical convergent to $x$, then we denote it by the notation $x_{n} \stackrel{r-st}{\longrightarrow} x$ and $x$ is said to be a rough statistical limit point (or $r-st$ limit point) of $\{x_n\}$. The set of all $r-st$ limit points of a sequence $\{x_n\}$ is said to be the $r-st$ limit set. We denote it by $st-LIM^r x_n$, i.e., $st-LIM^{r}x_{n}= \left\{x \in X : x_{n} \stackrel{r-st}{\longrightarrow} x \right\}$. For the roughness degree $r>0$, the $r-st$ limit may not be unique.
\begin{theorem}
    Every rough convergent sequence in an $S$-metric space $(X,S)$ is rough statistically convergent.
\end{theorem}
\begin{proof}
Let a sequence $\{\xi_n\}$ be rough convergent to $\xi$ in $(X,S)$.
Let $\varepsilon>0$ be arbitrary.
Then for $\varepsilon>0$, there exists $m \in \mathbb N$ such that $ S(\xi_n,\xi_n,\xi) < r+\varepsilon$, $\forall n \geq m$.
Then the set $A=\{n \in \mathbb N : S(\xi_n,\xi_n,\xi) \geq r+\varepsilon\} \subset \{1,2,3,......,(m-1) \}$.
So, $A$ is a finite set, and therefore, $\delta(A)=0$.
Hence $\{\xi_n\}$ is rough statistically convergent in $(X,S)$.
    
\end{proof}

\begin{remark}
    The converse of the above theorem may not be true. The following example shows that a rough statistically convergent sequence may not be rough convergent in an $S$-metric space. 
\end{remark}

\begin{example}
  Let $X=\mathbb R^2$ and $||.||$ be the Euclidean norm on $X$, then $S(x,y,z)=||x-z||+||y-z||$, \ $\forall x,y,z \in X$ is an $S$-metric on $X$.
  Let a sequence $\{\xi_{n}\}$ be defined by 
  \begin{equation*}
    \  \xi_{n}= \begin{cases}
    (p,p), & \text { if $n=p^2$ for $p \in \mathbb N$ }, \\
    (0,0), & \text { if $n \neq p^2$ and $n$ is even}, \\
    (1,1), & \text { if $n \neq p^2$ and $n$ is odd. }
                 \end{cases}
  \end{equation*}
 Let $\varepsilon>0$ be arbitrary and let $\xi_1=(0,0)$. \\
 Then $A_1=\{n \in \mathbb N : S(\xi_n,\xi_n,\xi_1) \geq 2\sqrt{2}+\varepsilon\} \subset \{1^2,2^2,3^2,4^2,5^2.....\}=M$ (say).\\
 Again, let $\xi_2=(1,1)$. Then 
  $A_2=\{n \in \mathbb N : S(\xi_n,\xi_n,\xi_2) \geq 2\sqrt{2}+\varepsilon\} \subset M$. \\
 Since $\delta(M)=0$, so $\delta(A_1)=\delta(A_2)=0$.\\
 So, $\{\xi_{n}\}$ is rough statistical convergent to (0,0) and (1,1) of roughness degree $2\sqrt{2}$.\\
 But \begin{equation*}
     S(\xi_n,\xi_n,\xi_1) = 2|| \xi_n-\xi_1 || = \begin{cases}
         2\sqrt{2}p, & \text{if $n=p^2$},\\
         0, & \text{if $n \neq p^2$ and $n$ is even}, \\
         2\sqrt{2}, & \text{if $n \neq p^2$ and $n$ is odd}.
     \end{cases}
 \end{equation*}
 So, when $n=p^2$, there dose not exist any positive integer $n_0$ such that the condition $S(\xi_n,\xi_n,\xi_1)< r+ \varepsilon$ for all $n \geq n_0$ holds.
 Hence $\{\xi_{n}\}$ is not rough convergent to (0,0) of any roughness degree $r>0$.
 Similarly, it can be shown that  $\{\xi_{n}\}$ is not rough convergent to (1,1).
\end{example}

\begin{definition} \cite{SUK1}
    The diameter of a set $A$ in an $S$-metric space $(X,S)$ is defined by 
    \begin{center}
    $diam(A)$ = $sup$ $\{ S(x,x,y) : x, y\in A \}$.
\end{center}
\end{definition}

\begin{theorem}
Let $\{x_{n}\}$ be a $r$-statistical convergent sequence in $(X,S)$. Then the diameter of $st-LIM^r x_n$ is not greater than $3r$ i.e. $dim(st-LIM^r x_n) \leq 3r$.  
\end{theorem}

\begin{proof}
If possible, suppose that $diam(st-LIM^{r} x_{n}) > 3r$.
Then there exist elements $\xi,\eta \in st-LIM^{r} x_{n}$ such that $S(\xi,\xi,\eta)>3r$.
Take $ \varepsilon \in (0,\frac{S(\xi,\xi,\eta)}{3}-r)$. 
Since $\xi,\eta \in st-LIM^{r} x_{n}$, $\delta(A_1)=0$ and $\delta(A_2)=0$, where 
$ A_1= \{ n \in \mathbb N: S(x_n,x_n,\xi) \geq r+\varepsilon \}$ and
$ A_2= \{ n \in \mathbb N: S(x_n,x_n,\eta) \geq r+\varepsilon \}$.
So, from the property of natural density, we can write $d(A^{c}_1 \cap A^{c}_2)=1$.
Now, for all $n \in A^{c}_1 \cap A^{c}_2$
 \begin{equation*}
       \begin{split}
S(\xi,\xi,\eta) &\leq S(\xi,\xi,x_n)+S(\xi,\xi,x_n)+
S(\eta,\eta,x_n)\\
       &= S(x_n,x_n,\xi)+S(x_n,x_n,\xi)+S(x_n,x_n,\eta), \ \text{by Lemma[2.1]} \\
       &< (r+\varepsilon)+(r+\varepsilon)+(r+\varepsilon)\\
       &=3(r+\varepsilon)\\
       &=3r+3\varepsilon\\
       &=3r+3 \{\frac{S(\xi,\xi,\eta)}{3}-r \} \\ 
       &=3r+S(\xi,\xi,\eta)-3r \\
       &=S(\xi,\xi,\eta), \ \text{which is a contradiction}.    
       \end{split}
   \end{equation*}    
Hence the diameter of $st-LIM^{r} x_{n}$ is not greater than $3r$ i.e. $diam(st-LIM^{r} x_{n}) \leq 3r$. 
\end{proof}

\begin{theorem}
    If a sequence $\{x_{n}\}$ statistically converges to $x'$ in an $S$-metric space $(X,S)$, then $st-LIM^{r}x_{n}=B_{S}[x',r]$.
\end{theorem}

\begin{proof}
Let the sequence $\{x_{n}\}$ be statistically convergent to $x'$.
Let $\varepsilon >0$.
Then $\delta(A)=0$, where $A=\{n \in \mathbb N: S(x_n,x_n,x') \geq \frac{\varepsilon}{3} \}$. \\
 Let $y \in B_{S}[x',r]= \{ y \in X:S(y,y,x')\leq r \}$.
 Then using the  triangularity properties of $S$-metric space, we have for $n \in A^c$
 \begin{equation*}
    \begin{split}
 S(x_n,x_n,y) &\leq S(x_n,x_n,x')+S(x_n,x_n,x')+S(y,y,x')\\
              & < \frac{\varepsilon}{3}+\frac{\varepsilon}{3}+r\\
              & < r+\varepsilon 
    \end{split}
\end{equation*}
So $\{n \in \mathbb N: S(x_n,x_n,y) \geq r+\varepsilon\} \subset A$.
Therefore, $\delta(\{n \in \mathbb N: S(x_n,x_n,y) \geq r+ \varepsilon \})=0$.
So, $y \in st-LIM^{r}x_{n}$. Hence 
\begin{equation}
    B_{S}[x', r] \subset st-LIM^{r}x_{n}.
\end{equation}
Again, let $y \in st-LIM^{r}x_{n}$ and let $\varepsilon>0$ be arbitrary.
Then $ \delta(B)=0$, where $B=\{n \in \mathbb N: S(x_n,x_n,y) \geq r+\frac{\varepsilon}{3} \}$.
Now, let $ n \in A^c \cap B^c$
\begin{equation*}
    \begin{split}
S(x',x',y) & \leq S(x',x',x_n)+S(x',x',x_n)+S(y,y,x_n)\\
           & = S(x_n,x_n,x')+S(x_n,x_n,x')+S(x_n,x_n,y)\\
           & < \frac{\varepsilon}{3}+\frac{\varepsilon}{3}+r+\frac{\varepsilon}{3}\\
           & =(r+\varepsilon)
    \end{split}
\end{equation*}
So $S(y,y,x')<r+\varepsilon$. Since $\varepsilon>0$ is arbitrary, $S(y,y,x') \leq r$
So, $y \in B_{S}[x',r]$. Hence 
\begin{equation}
    st-LIM^{r}x_{n} \subset B_{S}[x', r]
\end{equation}
From (4.1) and (4.2), we get
$st-LIM^{r}x_{n}=B_{S}[x', r]$.
\end{proof}

\begin{theorem}
   Let $\{x_{n}\}$ be a $r$-statistical convergent sequence in an $S$-metric space $(X,S)$ and $\{\xi_{n}\}$ be a convergent sequence in $st-LIM^{r}x_{n}$ converging to $\xi$. Then $\xi$ must belongs to $st-LIM^{r}x_{n}$.
\end{theorem}
\begin{proof}
If $st-LIM^{r}x_{n}=\phi$, then there is nothing to prove.
Let us consider the case when $st-LIM^{r}x_{n} \neq \phi$.
Let $\{\xi_{n}\}$ be a sequence in $st-LIM^{r}x_{n}$ such that $\{\xi_{n}\}$ converges to $\xi$.
Let $\varepsilon>0$ be arbitrary. So for $\varepsilon>0$, there exists $n_1 \in \mathbb N$ such that $S(\xi_n,\xi_n,\xi) < \frac{\varepsilon}{3}$, \ $\forall n \geq n_1$.
Now, let us choose an $n_0 \in \mathbb N$ such that $n_0 > n_1$.
Then we can write $S(\xi_{n_0},\xi_{n_0},\xi) < \frac{\varepsilon}{3}$.
Since $\{\xi_{n}\} \subset st-LIM^{r}x_{n}$, so we have $\xi_{n_0} \in st-LIM^{r}x_{n} $ and $\delta(\{n \in \mathbb N: S(x_n,x_n,\xi_{n_0}) \geq r+ \frac{\varepsilon}{3}\})=0$.
Now, we show that 
\begin{equation}
\{n \in \mathbb N: S(x_n,x_n,\xi) < r+ \varepsilon\} \supseteq \{n \in \mathbb N: S(x_n,x_n,\xi_{n_0})< r+ \frac{\varepsilon}{3}\}   
\end{equation}
Let $k \in \{n \in \mathbb N: S(x_n,x_n,\xi_{n_0}) < r+ \frac{\varepsilon}{3}\}$.\\
Then $S(x_k,x_k,\xi_{n_0}) < r+ \frac{\varepsilon}{3}$.
So, we can write
\begin{equation*}
    \begin{split}
         S(\xi,\xi,x_k) & \leq S(\xi,\xi,\xi_{n_0})+S(\xi,\xi,\xi_{n_0})+S(x_k,x_k,\xi_{n_0})\\
         & =S(\xi_{n_0},\xi_{n_0},\xi)+S(\xi_{n_0},\xi_{n_0},\xi)+S(x_k,x_k,\xi_{n_0}) \\
         & < \frac{\varepsilon}{3}+\frac{\varepsilon}{3}+r+ \frac{\varepsilon}{3} \\
         & = r+\varepsilon
    \end{split}
\end{equation*}
Hence $k \in \{n \in \mathbb N: S(\xi,\xi,x_n) < r+ \varepsilon\} = \{n \in \mathbb N: S(x_n,x_n,\xi) < r+ \varepsilon\}$.
Therefore, (3.3) holds.
Since the set on the right-hand side of (4.3) has natural density 1 and so, the natural density of the set on the left-hand side of (4.3) is equal to 1. So, we get 
$\delta(\{n \in \mathbb N :S(x_n,x_n,\xi) \geq r+ \varepsilon \})=0$.
Hence $\xi \in st-LIM^{r}x_{n}$. 
Therefore, the theorem follows.
\end{proof} 

\begin{theorem}
Every $S$-metric space $(X,S)$ is first countable.
\end{theorem}

\begin{proof}
Let $\tau(S)$ be the topology generated by the base $v$ where $v=\{ B_S(x,\varepsilon): x \in X, \varepsilon>0 \}$.
Let us consider $u=\{ B_S(x,\frac{1}{p}): x \in X, p \in \mathbb N \}$.
Then $u$ is a basis for $\tau(S)$.
For let, $A \in \tau(S)$ and $x \in A$ be arbitrary element. Then, since $v$ is a basis for $\tau(S)$, there exists $\varepsilon>0$ such that $x \in B_S(x, \varepsilon) \subset A$. Choose $p \in \mathbb N$, so that $\frac{1}{p}<\varepsilon$. Then $B_S(x,\frac{1}{p}) \subset B_S(x, \varepsilon)$. Thus $x \in B_S(x,\frac{1}{p}) \subset B_S(x, \varepsilon) \subset A$. So $u$ forms a basis for $\tau(S)$. Since $u$ is countable, therefore $(X,S)$ is first countable.
\end{proof}

\begin{corollary}
Let $\{x_{n}\}$ be a $r$-statistical convergent sequence in an $S$-metric space $(X,S)$. Then $st-LIM^{r}x_{n}$ is a closed set for any roughness degree $r \geq 0$.
\end{corollary}
\begin{proof}
 Since the $S$-metric space $(X,S)$ is first countable, the result follows directly from Theorem (4.4).   
\end{proof}

\begin{theorem}
Let $(X,S)$ be an $S$-metric space. Then a sequence $\{x_{n}\}$ is statistically bounded in $(X,S)$ if and only if there exists a non-negative real number $r$ such that $st-LIM^{r}x_{n} \neq \phi$.     
\end{theorem}

\begin{proof}
Let $p \in X$ be a fixed element in $X$.
Since the sequence $\{x_{n}\}$ is statistically bounded in $(X,S)$, so there exists a positive real number $M$ such that $\delta(\{ n \in \mathbb N : S(x_n,x_n,p) \geq M \})=0$.
Let $\varepsilon>0$ be arbitrary and $r=M$.
Then $M<r+\varepsilon$. So we have the inclusion
\begin{equation}
 \{ n \in \mathbb N : S(x_n,x_n,p) < r+\varepsilon \} \supset \{ n \in \mathbb N : S(x_n,x_n,p) < M \}    
\end{equation}
Hence 
$\{ n \in \mathbb N : S(x_n,x_n,p) \geq r+\varepsilon \} \subset \{ n \in \mathbb N : S(x_n,x_n,p) \geq M \}$. 
Since $\delta(\{ n \in \mathbb N : S(x_n,x_n,p) \geq M \})=0$, it follows that $\delta(\{ n \in \mathbb N : S(x_n,x_n,p) \geq r+\varepsilon \})=0$. Therefore $p \in st-LIM^{r}x_{n} $ i.e. $st-LIM^{r}x_{n} \neq \phi$. \\

Conversely, let $st-LIM^{r}x_{n} \neq \phi$ and let $p$ be a $r$-limit of $\{x_{n}\}$. Then for $\varepsilon >0$, 
$\delta(\{ n \in \mathbb N : S(x_n,x_n,p) \geq r+\varepsilon \})=0$. Let $A=\{ n \in \mathbb N : S(x_n,x_n,p) \geq r+\varepsilon \}$ and $ M=r+\varepsilon$. Then $\delta(A)=0$ and if $ n \in A^c$, then $S(x_n,x_n,p)<r+\varepsilon=M$.
So, $ n \in \{ n \in \mathbb N:S(x_n,x_n,p) < M \} $.
This implies $ A^c \subset \{ n \in \mathbb N:S(x_n,x_n,p) < M \}$.
So, $ \{ n \in \mathbb N:S(x_n,x_n,p) \geq M \} \subset A$ and hence $\delta(\{ n \in \mathbb N : S(x_n,x_n,p) \geq M \})=0$.
Therefore $\{x_{n}\}$ is statistically bounded. 
\end{proof}

\begin{theorem}
Let $\{ x_{n_{p}}\}$ be a subsequence of $\{x_{n}\}$ such that $\delta(\{ n_1, n_2,....... \})=1$, then $st-LIM^{r}x_{n} \subseteq st-LIM^{r}x_{n_{p}}$. 
\end{theorem}

\begin{proof}
Let  $\{ x_{{n}_{p}}\}$ be a subsequence of $\{x_{n}\}$ and $ x \in st-LIM^{r}x_{n}$. 
Let $ \varepsilon >0$ be arbitrary, then the set 
$A=\{ n \in \mathbb N : S(x_n,x_n,x) \geq r+\varepsilon \}$ has density zero. 
So, $\delta(A^c)=1$. 
Since the set $P= \{ n_1, n_2,....... \}$ has density 1, $A^c \cap P \neq \phi$.
For if $A^c \cap P= \phi$, then $P \subset A$ and so $\delta(P)=0$, a contradiction, since $\delta(A)=0$.
Let $n_k \in A^c \cap P$. Then $S(x_{{n}_{k}},x_{{n}_{k}},x) < r+\varepsilon$.
So, $\{ n_p \in P: S(x_{{n}_{p}},x_{{n}_{p}},x) \geq r+\varepsilon  \} \subset A \cup P^c$. This implies that $ \delta( \{ n_p \in P: S(x_{{n}_{p}},x_{{n}_{p}},x) \geq r+\varepsilon  \} )=0$, since $\delta(A \cup P^c) \leq \delta(A)+\delta(P^c)=0+0=0$. 
Therefore $ x \in st-LIM^{r}x_{n_{p}}$. Hence $st-LIM^{r}x_{n} \subseteq st-LIM^{r}x_{n_{p}}$.
\end{proof}

\begin{theorem}
Let $\{\xi_{n}\}$ and $\{\eta_{n}\}$ be two sequences in $(X,S)$ such that $S(\xi_{n},\xi_n,\eta_{n}) \longrightarrow 0$ as $n \longrightarrow \infty$. Then
$\{\xi_{n}\}$ is $r$-statistically convergent to $\xi$ if and only if $\{\eta_{n}\}$ is $r$-statistically convergent to $\xi$.   
\end{theorem}

\begin{proof}
 Let $\{\xi_{n}\}$ be $r$-statistically convergent to $\xi$. 
 Let $\varepsilon>0$. 
 Then for $\varepsilon>0$, $\delta(A)=0$, where $A=\{n \in \mathbb N: S(\xi_{n},\xi_n,\xi) \geq r+\frac{\varepsilon}{3} \}$. \\
 Since $S(\xi_{n},\xi_n,\eta_{n}) \longrightarrow 0$ as $n \longrightarrow \infty$, for $\varepsilon>0$, $ \exists \ k \in \mathbb N$ such that $S(\xi_{n},\xi_n,\eta_{n})\leq \frac{\varepsilon}{3}$, when $n \geq k$. \\
Since $\delta(A)=0$, $\delta(A^c)=1$.
Again, $\delta( \{1,2,......k \})=0$, so $ \delta( \{1,2,......k \}^c)=1$. 
So, $A^c \cap \{1,2,......k \}^c \neq \phi $.
If $ n \in A^c \cap \{1,2,......k \}^c $, then
\begin{equation*}
    \begin{split}
S(\eta_n,\eta_n,\xi) & \leq S(\eta_n,\eta_n,\xi_n)+S(\eta_n,\eta_n,\xi_n)+S(\xi,\xi,\xi_n) \\
    & = S(\xi_n,\xi_n,\eta_n)+S(\xi_n,\xi_n,\eta_n)+S(\xi_n,\xi_n,\xi) \\
    & < \frac{\varepsilon}{3} +\frac{\varepsilon}{3} +(r+ \frac{\varepsilon}{3}) \\
    & = r + \varepsilon 
    \end{split}
\end{equation*}
Hence $ A^c \cap \{1,2,......k \}^c \subset \{n \in \mathbb N : S(\eta_n,\eta_n,\xi) < r+\varepsilon \} $.
So, $ \{n \in \mathbb N : S(\eta_n,\eta_n,\xi) \geq r+\varepsilon \} \subset (A^c \cap \{1,2,......k \}^c)^c = A \cup \{1,2,......k \} $.
Now, since $ \delta(A \cup \{1,2,......k \}) \leq \delta(A)+\delta(\{1,2,......k \}) = 0+0=0$, so
 $\delta(\{n \in \mathbb N: S(\eta_n,\eta_n,\xi) \geq r+\varepsilon \})=0$.
Therefore, $\{\eta_{n}\}$ is $r$-statistical convergent to $\xi$.

Converse part is similar.
\end{proof}

\begin{definition} (cf. \cite{AYTER1})
Let $(X,S)$ be an $S$-metric space. Then $c \in X$ is called a statistical cluster point of a sequence $\{x_{n}\}$ in $(X,S)$ if for every $\varepsilon>0$, $\delta(\{n \in \mathbb N: S(x_n,x_n,c) < \varepsilon \}) \neq 0$.
\end{definition}

\begin{theorem}
 Let $\{x_{n}\}$ be a sequence in an $S$-metric space $(X, S)$. If $c$ is a cluster point of $\{x_{n}\}$, then $st-LIM^{r}x_{n}  \subset B_{S}[c,r]$ for some $r >0$.
\end{theorem}

\begin{proof}
Let $\varepsilon>0$ and let  $x \in st-LIM^{r}x_{n}$.
Then for $\varepsilon > 0$, $\delta(B_1)=0$, where $B_1=\{n \in \mathbb{N}: S(x_n,x_n,x) \geq r+ \frac{\varepsilon}{3} \}$. 
Again, since $c$ is a cluster point of $\{x_{n}\}$, for the same $\varepsilon > 0$, $\delta(B_2) \neq 0$, where $B_2=\{n \in \mathbb{N}: S(x_n,x_n,c) < \frac{\varepsilon}{3} \}$.
Now, let $k \in B_{1}^c \cap B_{2}$, then 
 $S(x_k,x_k,x) < r+ \frac{\varepsilon}{3}$ and $S(x_k,x_k,c) < \frac{\varepsilon}{3}$.
Therefore, we can write
\begin{equation*}
    \begin{split}
        S(c,c,x) & \leq S(c,c,x_k)+S(c,c,x_k)+S(x,x,x_k) \\
                 & = S(x_k,x_k,c)+S(x_k,x_k,c)+S(x_k,x_k,x) \\
                 & < \frac{\varepsilon}{3} +\frac{\varepsilon}{3} +(r+ \frac{\varepsilon}{3}) \\
                 & = r + \varepsilon
    \end{split}
\end{equation*}
Since $\varepsilon>0$ is arbitrary, $S(c,c,x) \leq r$ and hence $S(x,x,c) \leq r$.
Therefore $x \in B_{S}[c,r]$.
Hence $st-LIM^{r}x_{n}  \subset \overline B_{S}[c,r]$ holds for some $r >0$.

\end{proof}

\subsection*{Acknowledgements}
The first author is thankful to The University of Burdwan for the grant of Senior Research Fellowship (State Funded) during the preparation of this paper. Both authors are also thankful to DST, Govt of India for providing FIST project to the dept. of Mathematics, B.U. \\


\begin{thebibliography}{99}

\bibitem{RA}
R. Abazari, \textit{Statistical Convergence in g-Metric Spaces}, Filomat 36(5) (2022), 1461–1468, https://doi.org/10.2298/FIL2205461A.

\bibitem{AYTER1}
S. Aytar, 
\textit{Rough statistical convergence,} 
Numer. Funct. Anal. Optim. {29(3-4)} (2008), 291-303.

\bibitem{AYTER2}
S. Aytar, 
\textit{The rough limit set and the core of a real Sequence,} 
Numer. Funct. Anal. Optim. {29(3-4)} (2008), 283-290.

\bibitem{AKBAB}
A. K. Banerjee, A. Banerjee, \textit{A study on I-Cauchy sequences and I-divergence in S-metric spaces}, Malaya Journal of Matematik, Vol. 6, No. 2, 326-330, 2018.

\bibitem{AKB}
A. K. Banerjee and A. Dey, 
\textit{Metric Spaces and Complex Analysis,}  
New Age International (P) Limited, Publication, 
ISBN-10: 81-224-2260-8, ISBN-13: 978-81-224-2260-3.

\bibitem{RMROUGH}
A. K. Banerjee, R. Mondal,
\textit{Rough convergence of sequences in a cone metric space,} J. Anal. 27(3-4) (2019), 1179–1188.

\bibitem{AKB,AP}
A. K. Banerjee and A. Paul, \textit{On $I$ and $I^*$-Cauchy conditions in $C^*$-algebra valued metric spaces}, Korean J. Math. 29(3) (2021) 621-629, http://dx.doi.org/10.11568/kjm.2021.29.3.621.

\bibitem{AKB,MP}
A. K. Banerjee and M. Paul, \textit{Strong $I^k$- convergence in probabilistic metric spaces}, Iranian Journal of Mathematical Sciences and Informatics, 17(2) (2022), 273-288, DOI: 10.52547/ijmsi.17.2.273.

\bibitem{SUK2}
A. K. Banerjee and S. Khatun, 
\textit{Rough convergence of sequences in a partial metric space,}
arXiv: 2211.03463, 2022.

\bibitem{SUK4}
A. K. Banerjee and S. Khatun, 
\textit{Rough statistical convergence of sequences in a partial metric space,}
arXiv: 2402.14452, 2024 (to be appear).

\bibitem{DR}
S. Debnath and D. Rakshit, 
\textit{Rough convergence in metric spaces}, 
 In: Dang, P., Ku, M., Qian, T., Rodino, L. (eds) New Trends in Analysis and Interdisciplinary Applications. Trends in Mathematics(). Birkhäuser, Cham.$https://doi.org/10.1007/978-3-319-48812-7_57$. 

 
\bibitem{DG1}
D. Georgiou, A. Megaritis, G. Prinos and F. Sereti,
\textit{On statistical convergence of sequences of closed sets in metric spaces}, Mathematica Slovaca, 71(2) (2021), 409–422.

\bibitem{DG2}
D. Georgiou, G. Prinos and F. Sereti,
\textit{Statistical and Ideal Convergences in Topology}, Mathematics, 11(3) (2023), p. 663.

\bibitem{LLG}
Kedian Li, Shou Lin, Ying Ge, On statistical convergence in cone metric spaces, Journal of Topology and its Applications 196 (2015) 641-651.

\bibitem{KM}
K. Menger, \textit{Statistical metrics}, Proceedings of the National Academy of Sciences of the United States of America, 28, (1942), 535-537.

\bibitem{MJS}
 Z. Ma, L. Jiang, and H. Sun, \textit{$C^*$-algebra-valued metric spaces and related fixed point theorems}, Fixed Point Theory Appl. 206 (2014), 2014.

\bibitem{HF}
H. Fast, \textit{Sur la convergence ststistique} ,Colloq. Math. 2(1951), 241-244.

 \bibitem{NH}
 N. Hossain and A. K. Banerjee,
 \textit{Rough I-convergence in intuitionistie fuzzy normed space},
 Bulletin of Mathematical Analysis and Application,
 14(4) (20220), 1-10.

\bibitem{PMROUGH1}
P. Malik, and M. Maity, 
\textit{On rough convergence of double sequence in normed linear spaces,}
Bull. Allahabad Math. Soc. {28(1)} (2013), 89-99.

\bibitem{PMROUGH2}
P. Malik, and M. Maity, 
\textit{On rough statistical convergence of double sequences in normed linear spaces,} 
Afr. Mat. 27(2016), 141-148.

\bibitem{SUK1}
R. Mondal and S. Khatun, \textit{Rough convergence of sequences in an $S$ metric space}, Palestine Journal of Mathematics, 13(1) (2024), 316-322.

\bibitem{PHU}
H. X. Phu, 
\textit{Rough convergence in normed linear spaces,} 
Numer. Funct. Anal. Optim.
22(1-2) (2001), 199-222.

\bibitem{PHU1}
H. X. Phu, 
\textit{Rough convergence in infinite dimensional normed spaces,} 
Numer. Funct. Anal. Optim. 24(2-3) (2003), 285-301.

\bibitem{HS}
H. Steinhaus , \textit{Sur la convergence ordinaire et la convergence asymptotique} ,Colloq. Math. 2 (1951) 73-74.

\bibitem{SSA}
S. Sedghi, N. Shobe, A. Aliouche, \textit{A generalization of fixed point theorems in $S$-metric spaces}, Matematicki Vesnik 64(2012), 258-266.

\bibitem{TS}
T. Salat, \textit{On statistically convergent sequence of real numbers}, Mathematica Slovaca, 30(2) (1980), 139-150.


\end{thebibliography}
\end{document}